\documentclass[11pt]{amsart}
\usepackage{mathrsfs}
\usepackage{latexsym,amssymb, amsmath,amscd,amsthm,amsxtra}
\usepackage{textcomp}
\usepackage{float}
\usepackage[usenames]{color}
\usepackage{graphicx}
\usepackage{amssymb}
\usepackage{amscd}
\usepackage{makecell}
\theoremstyle{plain}
\newtheorem{thm}{Theorem}[section]
\newtheorem{lemma}[thm]{Lemma}

\newtheorem{conj}[thm]{Conjecture}

\theoremstyle{definition}

\newtheorem{rmk}[thm]{Remark}
\newtheorem{example}[thm]{Example}

\def\dim{\mathop{\hbox {dim}}\nolimits}

\newcommand{\fra}{\mathfrak{a}}

\newcommand{\frg}{\mathfrak{g}}
\newcommand{\frh}{\mathfrak{h}}

\newcommand{\frk}{\mathfrak{k}}
\newcommand{\frl}{\mathfrak{l}}

\newcommand{\frp}{\mathfrak{p}}
\newcommand{\frq}{\mathfrak{q}}

\newcommand{\frt}{\mathfrak{t}}
\newcommand{\fru}{\mathfrak{u}}

\newcommand{\bbC}{\mathbb{C}}

\newcommand{\bbR}{\mathbb{R}}

\newcommand{\bbZ}{\mathbb{Z}}

\newcommand{\caL}{\mathcal{L}}

\newcommand{\caR}{\mathcal{R}}

\begin{document}

\title{A non-vanishing criterion for Dirac cohomology}

\author{Chao-Ping Dong}
\address[Dong]{School of Mathematical Sciences, Soochow University, Suzhou 215006,
P.~R.~China}
\email{chaopindong@163.com}

\abstract{This paper gives a criterion for the non-vanishing of the Dirac cohomology of $\caL_S(Z)$, where  $\caL_S(\cdot)$ is the cohomological induction functor, while the inducing module $Z$ is irreducible, unitarizable, and in the good range. As an application, we give a formula counting the number of strings in the Dirac series.
Using this formula, we classify all the irreducible unitary representations of $E_{6(2)}$ with non-zero Dirac cohomology. Our calculation continues to support Conjecture 5.7' of Salamanca-Riba and Vogan \cite{SV}. Moreover, we find more unitary representations for which cancellation happens between the even part and the odd part of their Dirac cohomology.
}

\endabstract



\maketitle


\section{Introduction}

A formula for the Dirac cohomology of cohomologically induced modules has been given in Theorem B of \cite{DH}. However, even if the inducing irreducible unitary module $Z$ has non-zero Dirac cohomolgy and lives in the good range, we do not know whether the cohomologically induced module $\caL_S(Z)$ has non-zero Dirac cohomology or not. The first aim of this note is to fix this problem.

We start with a complex connected  simple \emph{linear} group $G(\bbC)$ which has finite center.
Let $\sigma: G(\bbC) \to G(\bbC)$ be a \emph{real form} of $G(\bbC)$. That is, $\sigma$ is an antiholomorphic Lie group automorphism such that $\sigma^2={\rm Id}$. Let $\theta: G(\bbC)\to G(\bbC)$ be the involutive algebraic automorphism of $G(\bbC)$ corresponding to $\sigma$ via Cartan theorem (see Theorem 3.2 of \cite{ALTV}). Put $G=G(\bbC)^{\sigma}$ as the group of real points. Note that $G$ must be in the \emph{Harish-Chandra class} \cite{HC}. That is,
\begin{itemize}
\item[(a)] $G$ has only a finite number of connected components;
\item[(b)] the derived group $[G, G]$ has finite center;
\item[(c)]  the adjoint action $Ad(g)$ of any $g\in G$ is an inner automorphism of  $\frg=(\frg_0)_{\bbC}$.
\end{itemize}
Denote by $K(\bbC):=G(\bbC)^{\theta}$, and put $K:=K(\bbC)^{\sigma}$. Denote by $\frg_0={\rm Lie}\,(G)$, and let
$$
\frg_0=\frk_0\oplus \frp_0
$$
be the  Cartan decomposition corresponding to $\theta$ on the Lie algebra level.

Choose a maximal torus $T_f$ of $K$. Let $\frt_{f, 0}={\rm Lie}\, (T_f)$ and put $\fra_{f, 0}=Z_{\frp_0}(\frt_{f, 0})$. Let $A_f$ be the corresponding analytic subgroup of $G$. Then $H_f=T_f A_f$ is a $\theta$-stable fundamental Cartan subgroup of $G$. As usual, we drop the subscript for the complexification. For instance,
$$
\frh_f=\fra_f \oplus \frt_f
$$
is the Cartan decomposition of the complexified Lie algebra of $H_f$.

We fix a positive root system $\Delta^+(\frk, \frt_f)$ once for all. Denote by $\rho_K$ the half sum of roots in $\Delta^+(\frk, \frt_f)$. Then there are $s$ ways of choosing positive roots systems of $\Delta(\frg, \frt_f)$  containing the fixed $\Delta^+(\frk, \frt_f)$. Here
\begin{equation}\label{s}
 s=\frac{\#W(\frg, \frt_f)}{\#W(\frk, \frt_f)},
\end{equation}
 where $W(\frk, \frt_f)$ (resp., $W(\frg, \frt_f)$) is the Weyl group of the root system $\Delta(\frk, \frt_f)$ (resp., $\Delta(\frg, \frt_f)$).
We will refer to a $K$-type by one of its highest weights.

Recall that any $\theta$-stable parabolic subalgebra $\mathfrak{q}$ of $\frg$ can be obtained by choosing an element $H\in i\frt_{f,0}$, and setting $\mathfrak{q}$ as the sum of  non-negative eigenspaces of ${\rm ad}(H)$. The Levi subalgebra $\frl$ of $\mathfrak{q}$ is the sum of  zero eigenspaces of ${\rm ad}(H)$. Note that $\frl$ and $\mathfrak{q}$ are both $\theta$-stable since $H$ is so. Put $L=N_G(\mathfrak{q})$. Then $L\cap K$ is a maximal compact subgroup of $L$.

We choose a positive root system $\Delta^+(\frg, \frt_f)$ so that $\Delta(\fru, \frt_f)\subseteq \Delta^+(\frg, \frt_f)$. Let $\Delta^+(\frg, \frt_f)$ be the union of the fixed $\Delta^+(\frk, \frt_f)$ and $\Delta^+(\frp, \frt_f)$. We denote the half sum of roots in $\Delta(\fru)$ as $\rho(\fru)$. Put
$$
\Delta(\fru\cap\frk)=\Delta(\fru)\cap \Delta^+(\frg, \frt_f), \quad \Delta(\fru\cap\frp)=\Delta(\fru)\cap \Delta^+(\frp, \frt_f),
$$
and denote the half sum of roots in them by $\rho(\fru\cap\frk)$ and $\rho(\fru\cap\frp)$, respectively.
Note that
\begin{equation}\label{rhoukp}
\rho(\fru)=\rho(\fru\cap\frk)+\rho(\fru\cap\frp).
\end{equation}
Put
$$
\Delta^+(\frl, \frt_f)=\Delta(\frl, \frt_f) \cap \Delta^+(\frg, \frt_f), \quad \Delta^+(\frl\cap\frk, \frt_f)=\Delta(\frl, \frt_f) \cap \Delta^+(\frk, \frt_f).
$$
Denote the half sum of roots in $\Delta^+(\frl\cap\frk, \frt_f)$ by $\rho_{L\cap K}$.

Cohomological induction functors lifts an $(\frl, L\cap K)$-module $Z$ to $(\frg, K)$-modules $\caL_j(Z)$ and $\caR^j(Z)$, where $j$ are some non-negative integers.
The interesting thing usually happens at the middle degree $S:={\rm dim} \,(\fru\cap\frk)$. The results that we need about cohomological induction will be summarized in Theorem \ref{thm-Vogan-coho-ind}.

We will recall Dirac cohomology in Section 2. This notion was introduced by Vogan \cite{Vog97}. Motivated by his conjecture on Dirac cohomology proven by Huang and Pand\v zi\'c \cite{HP}, we say that a weight $\Lambda\in\frh_f^*$ satisfies the \textbf{Huang-Pand\v zi\'c condition} (\emph{HP condition} for short henceforth) if
\begin{equation}\label{Lambda-HP}
\{\delta-\rho_n^{(j)}\}+\rho_K=w \Lambda,
\end{equation}
where $\delta$ is any highest weight of some $K$-type, $0\leq j\leq s-1$, and $w\in W(\frg, \frt_f)$. Note that if $\Lambda$ satisfies the HP condition, then it must be \emph{real} in the sense of Definition 5.4.1 of \cite{V1}. That is, $\Lambda\in i\frt_{f,0}^* +\fra_{f,0}^*$.

\begin{thm}\label{thm-main}
Let $G$ be a simple linear real Lie group in the Harish-Chandra class. Let $Z$ be an irreducible unitary $(\frl, L\cap K)$-module with infinitesimal character $\lambda_L\in i\frt_{f,0}^*$ such that $\lambda_L$ is $\Delta^+(\frl\cap\frk, \frt_f)$-dominant. Assume that $\lambda_L+\rho(\fru)$ is good. That is,
$$\langle \lambda_L+\rho(\fru), \alpha^{\vee}\rangle>0, \quad \forall\alpha\in\Delta(\fru, \frt_f).
$$
Assume moreover that $\lambda_L +\rho(\fru)$ satisfies the HP condition. Then $H_D(\caL_S(Z))$ is non-zero if and only if $H_D(Z)$ is non-zero.
\end{thm}

\begin{rmk}\label{rmk-thm-DI-SUpq-Aqlambda}
(a)\ If $\lambda_L+\rho(\fru)$, a representative vector of the infinitesimal character of $\caL_S(Z)$, does not satisfy the HP condition, then $H_D(\caL_S(Z))$ must be zero in view of Theorem \ref{thm-HP}.

\smallskip

\noindent (b) If we further assume $G$ to be connected, then $L$ is connected as well. In this case, the proof will say that $\gamma_L\mapsto \gamma_L +\rho(\fru\cap\frp)$ is a multiplicity-preserving bijection from the $\widetilde{K_L}$-types of $H_D(Z)$ to the $\widetilde{K}$-types of $H_D(\caL_S(Z))$. Here $K_L:=K\cap L$, and $\widetilde{K_L}$ is the pin double covering group of $K_L$. This completely extends Theorem 6.1 of \cite{D13} to real linear groups.

\smallskip

\noindent (c)\ Example \ref{exam-weakly-good} will tell us that the good range condition in the above theorem can \emph{not} be weakened, say, to be weakly good.
\end{rmk}

Recall that in \cite{D20}, a finiteness result has been given on the classification of $\widehat{G}^d$---the set of all equivalence classes of irreducible unitary $(\frg, K)$-modules with non-zero Dirac cohomology.
As suggested by Huang, we call the set $\widehat{G}^d$ the \emph{Dirac series} of $G$. Theorem \ref{thm-main} allows us to completely determine the number of strings in $\widehat{G}^d$, see Section \ref{sec-number-strings}. Using this formula, we classify the Dirac series for the group $E_{6(2)}$ as follows.

\begin{thm}\label{thm-EII}
The set $\widehat{E_{6(2)}}^d$ consists of $56$ fully supported scattered representations (see Section \ref{sec-appendix}) whose spin lowest $K$-types are all unitarily small, and $576$ strings of representations. Each spin-lowest $K$-type of any Dirac series representation of $E_{6(2)}$ occurs with multiplicity one.
\end{thm}

In the above theorem, the notion of fully supported scattered representation will be recalled in Section \ref{sec-scattered-FS-scattered}, that of spin-lowest $K$-type will be recalled in Section 6, and that of unitarily small $K$-type comes from \cite{SV}. We sort the statistic $\|\nu\|^2$ for these $56$ fully supported scattered representations as follows:
$$
4.5, 8, \left(8.5\right)^{\underline 2},10^{\underline 8},
\left(10.5\right)^{\underline 4}, 13^{\underline 6}, 14^{\underline 2}, 15^{\underline 2}, 17^{\underline 4}, \left(17.5\right)^{\underline 4}, 18^{\underline 2}, 29^{\underline 8},
29.5, 30^{\underline 9}, 42, 78,
$$
where $a^{\underline k}$ means that the value $a$ occurs $k$ times. Note that the original Helgason-Johnson bound (see \cite{HJ}) $\|\rho(G)\|^2=78$ is attained at the trivial representation, while the sharpened Helgason-Johnson bound (see \cite{D21}) $42$ is attained at the minimal representation, namely, the first entry of Table \ref{table-EII-111011}. Numerically, one sees that there is still a remarkable gap between $30^{\underline 9}$ and $42$. We will pursue this later.

Among the above $56$ fully supported scattered members of $\widehat{E_{6(2)}}^d$, cancellation happens within the Dirac cohomology for $10$ of them when passing to \emph{Dirac index}. Recall that Dirac index of $\pi$ is defined as the following virtual $\widetilde{K}$ module:
\begin{equation}\label{DI}
{\rm DI}(\pi) = H_D^+(\pi) - H_D^-(\pi).
\end{equation}
Here $H_D^+(\pi)$ (resp., $H_D^-(\pi)$) is the even (resp., odd) part of $H_D(\pi)$. See \cite{MPVZ}.

The first instance of the cancellation phenomenon is recorded in Example 6.3 of \cite{DDY} on \texttt{F4\_s},
which disproves Conjecture 10.3 of \cite{H15}.  It is worth noting that \texttt{F4\_s} is also a \emph{quaternionic} real form, a notion raised by Wolf \cite{Wo61} in 1961. See also Appendix C of Knapp \cite{Kn}. Let us sharpen Conjecture 10.3 of \cite{H15} to be the following one.

\begin{conj} Further assume that $G$ is equal rank. Let $\pi$ be any irreducible unitary $(\frg, K)$ module such that $H_D(\pi)$ is non-zero. Then the Dirac index of $\pi$ must vanish if ${\rm Hom}_{\widetilde{K}}(H_D^+(\pi), H_D^-(\pi))\neq 0$.
\end{conj}

The above conjecture asserts that there should be \emph{dichotomy} among the spin LKTs whenever cancellation happens.
By the way, our current calculations and those in \cite{DDH,DDY} lead us to make Conjecture \ref{conj-unitarity-from-G-to-L-HP}, which asserts that the reverse direction of an old theorem of Vogan (namely, item (iv) of Theorem \ref{thm-Vogan-coho-ind}) should hold under certain  additional restrictions.

The paper is organized as follows: necessary preliminaries will be collected in Section 2, the root system $\Delta(\frg, \frt_f)$ will be recalled in Section 3. We deduce Theorem \ref{thm-main} in Section 4, and give a formula counting the strings in $\widehat{G}^d$ in Section \ref{sec-number-strings}. Theorem \ref{thm-EII} will be proven in Section 6. The cancellation phenomenon will be studied in Section \ref{sec-can-EII}.
The special unipotent representations of $E_{6(2)}$ with non-vanishing Dirac cohomlogy will be discussed in Section \ref{sec-EII-su}. All the fully supported scattered members of $\widehat{E_{6(2)}}^d$ will be presented in Section \ref{sec-appendix} according to their infinitesimal characters.

\section{Preliminaries}\label{sec-pre}

We continue with the notation in the introduction, and collect necessary preliminaries in this section. Note  that $G$ must be linear.

\subsection{Dirac cohomology of cohomologically induced modules}
We fix a nondegenerate invariant symmetric bilinear form $B$ on $\frg_0$, which is positive definite on $\frp_0$ and negative definite on $\frk_0$. Its extensions/restrictions to $\frg$, $\frk_0$, $\frp_0$, etc., will also be denoted by the same symbol.

Fix
an orthonormal basis $Z_1, \dots, Z_n$ of $\frp_0$ with respect to the inner product induced by
$B$. Let $U(\frg)$ be the
universal enveloping algebra of $\frg$ and let $C(\frp)$ be the
Clifford algebra of $\frp$ (with respect to $B$). The Dirac operator
$D\in U(\frg)\otimes C(\frp)$ is defined by Parthasarathy \cite{Pa} as
$$D=\sum_{i=1}^{n}\, Z_i \otimes Z_i.$$
It is easy to check that $D$ does not depend on the choice of the
orthonormal basis $\{Z_i\}_{i=1}^n$ and it is $K$-invariant for the diagonal
action of $K$ given by adjoint actions on both factors.

Let $\widetilde{K}$ be pin covering group of $K$. That is, $\widetilde{K}$ is the subgroup of $K\times \text{Pin}(\frp_0)$ consisting of all pairs $(k, s)$ such that $\text{Ad}(k)=p(s)$, where $\text{Ad}: K\rightarrow \text{O}(\frp_0)$ is the adjoint action, and $p: \text{Pin}(\frp_0)\rightarrow \text{O}(\frp_0)$ is the pin double covering map.
If $X$ is a
($\frg$, $K$)-module, and if ${\rm Spin}_G$ denotes a spin module for
$C(\frp)$, then $U(\frg)\otimes C(\frp)$ acts on $X\otimes {\rm Spin}_G$ in
the obvious fashion, while $\widetilde{K}$ acts  on $X$
through $K$ and on ${\rm Spin}_G$ through the pin group
$\text{Pin}(\frp_0)$. Now the Dirac operator acts on $X\otimes {\rm Spin}_G$, and the Dirac
cohomology of $X$ is defined
as the $\widetilde{K}$-module
\begin{equation}\label{def-Dirac-cohomology}
H_D(X)=\text{Ker}\, D/ (\text{Im} \, D \cap \text{Ker} D).
\end{equation}
We embed $\frt_f^{*}$ as a subspace of $\frh_f^{*}$ by setting  the linear functionals on $\frt_f$ to be zero on $\fra_f$.
The following result slightly extends Theorem 2.3 of Huang and Pand\v zi\'c \cite{HP} to disconnected groups.

\begin{thm}\label{thm-HP} \emph{(Theorem A of \cite{DH})}
Let $G$ be a real
reductive Lie group  in Harish-Chandra class.
Let $X$ be an irreducible ($\frg$, $K$)-module with infinitesimal character $\Lambda$.
Suppose that  $\widetilde{\delta}$ is an irreducible $\widetilde{K}$-module in the Dirac cohomology $H_D(X)$
with a highest weight $\mu$. Then $\Lambda$ is conjugate to $\mu+\rho_K$
under the action of the Weyl group $W(\frg,\frh_f)$.
\end{thm}

A formula for the Dirac cohomology of cohomologically induced modules in the weakly good range has been given in \cite{DH}. See also \cite{Pan}. To state it, assume that the inducing $(\frl, L\cap K)$-module $Z$
has infinitesimal character $\lambda_L\in i\frt_{f,0}^*$ which is dominant for $\Delta^+(\frl\cap \frk, \frt_f)$. We say that $Z$ is \emph{weakly good} if
\begin{equation}
\langle \lambda_L + \rho(\fru), \alpha^{\vee} \rangle \geq 0, \quad \forall \alpha\in \Delta(\fru, \frt_f).
\end{equation}

\begin{thm}\label{thm-Vogan-coho-ind}
{\rm (\cite{Vog84} Theorems 1.2 and 1.3, or \cite{KV} Theorems 0.50 and 0.51)}
Suppose the admissible
 ($\frl$, $L\cap K$)-module $Z$ is weakly good.  Then we have
\begin{itemize}
\item[(i)] $\caL_j(Z)=\caR^j(Z)=0$ for $j\neq S$.
\item[(ii)] $\caL_S(Z)\cong\caR^S(Z)$ as ($\frg$, $K$)-modules.
\item[(iii)]  if $Z$ is irreducible, then $\caL_S(Z)$ is either zero or an
irreducible ($\frg$, $K$)-module with infinitesimal character $\lambda_L+\rho(\fru)$.
\item[(iv)]
if $Z$ is unitary, then $\caL_S(Z)$, if nonzero, is a unitary ($\frg$, $K$)-module.
\item[(v)] if $Z$ is in good range, then $\caL_S(Z)$ is nonzero, and it is unitary if and only if $Z$ is unitary.
\end{itemize}
\end{thm}

It is worth noting that the reverse direction of Theorem \ref{thm-Vogan-coho-ind}(iv) is not true in general. However, we suspect that it may hold in certain special cases. See Conjecture \ref{conj-unitarity-from-G-to-L-HP}.

Now we are able to state the aforementioned formula.

\begin{thm}{\rm (Theorem B of \cite{DH})}\label{thm-DH}
 Suppose that the irreducible unitary $(\frl, L\cap K)$-module $Z$ has infinitesimal character $\lambda_L\in i\frt_{f, 0}^*$ which is weakly good. Then there is a $\widetilde{K}$-module isomorphism
\begin{equation}\label{Dirac-coho}
H_D(\caL_S(Z)) \cong  \caL_S^{\widetilde{K}} (H_D(Z)\otimes \bbC_{-\rho(\fru\cap\frp)}).
\end{equation}
\end{thm}

In the setting of the above theorem, it is clear that when $H_D(\caL_S(Z))$ is non-zero, then $H_D(Z)$ must be non-zero. However, the other direction is unclear yet. Theorem \ref{thm-main} aims to fill this gap for linear groups.

\subsection{A very brief introduction of the software \texttt{atlas}}

Let $H(\bbC)$ be a \emph{maximal torus} of $G(\bbC)$.   Its \emph{character lattice} is the group of algebraic homomorphisms
$$
X^*:={\rm Hom}_{\rm alg} (H(\bbC), \bbC^{\times}).
$$
Choose a Borel subgroup $B(\bbC)\supset H(\bbC)$.
In the software \texttt{atlas} \cite{ALTV, At}, an irreducible $(\frg, K)$-module $\pi$ is parameterized by a \emph{final} parameter $p=(x, \lambda, \nu)$ via the Langlands classification. See Theorem 6.1 of \cite{ALTV}. Here $x$ is a $K(\bbC)$-orbit of the Borel variety $G(\bbC)/B(\bbC)$, $\lambda \in X^*+\rho$ and $\nu \in (X^*)^{-\theta}\otimes_{\bbZ}\bbC$. In such a case, the infinitesimal character of $\pi$ is
\begin{equation}\label{inf-char}
\frac{1}{2}(1+\theta)\lambda +\nu \in\frh^*,
\end{equation}
where $\frh$ is the Lie algebra of $H(\bbC)$.
Note that the Cartan involution $\theta$ now becomes $\theta_x$---the involution of $x$, which is given by the command \texttt{involution(x)} in \texttt{atlas}.

Among the three components of $p=(x, \lambda, \nu)$, the KGB element $x$ is hardest to understand. One can use the command \texttt{print\_KGB(G)} to see the rich information, and to identify which $K(\bbC)$-orbit of the Borel variety it is. For our study, the most relevant knowledge is the following special case of Theorem \ref{thm-Vogan-coho-ind}, rephrased in the language of \texttt{atlas} by Paul \cite{Paul}.

\begin{thm}\label{thm-Vogan-Paul-atlas}
 Let $p=(x, \lambda, \nu)$ be the \texttt{atlas} parameter of an irreducible $(\frg, K)$-module $\pi$.
Let $S(x)$ be the support of $x$, and $\frq(x)$ be the $\theta$-stable parabolic subalgebra given by the pair $(S(x), x)$, with Levi factor $L(x)$.
Then  $\pi$ is cohomologically induced, in the weakly good range,
from an irreducible $(\frl, L\cap K)$-module $\pi_{L(x)}$ with parameter $p_{L(x)}=(y, \lambda-\rho(\fru), \nu)$, where $y$ is the KGB element of $L(x)$ corresponding to the KGB element $x$ of $G$.
\end{thm}

Let $l$ be the rank of $G$. \texttt{atlas} labels the simple roots of $\Delta^+(\frg, \frh_f)$ as $\{0, 1, 2, \dots, l-1\}$, and the support of an KGB element $x$ is given by the command \texttt{support(x)}. We say that $x$ is \emph{fully supported} if \texttt{support(x)} equals to $[0, 1, 2, \dots, l-1]$.
Whenever $x$ is fully supported, we will have that $\frq(x)=\frg$. We say that the representation $p$ is \emph{fully supported} if its KGB element $x$ is so.

Let us illustrate Theorem \ref{thm-Vogan-Paul-atlas} and some basic \texttt{atlas} commands via a specific example. Some outputs will be omitted to save space.

\begin{example}\label{exam-atlas-basic}
Firstly, let us input the linear split real form of $F_4$ into $\texttt{atlas}$.
\begin{verbatim}
G:F4_s
Value: connected split real group with Lie algebra 'f4(R)'
#KGB(G)
Value: 229
support(KGB(G,228))
Value: [0,1,2,3]
\end{verbatim}
This group has 229 KGB elements in total, and the last one, i.e., \texttt{KGB(G, 228)}, is fully supported.
Indeed, as the following shows, there are $141$ fully supported KGB elements in total.
\begin{verbatim}
set FS=## for x in KGB(G) do if #support(x)=4 then [x] else [] fi od
#FS
Value: 141
\end{verbatim}
Now let us look at a KGB element which is not fully supported.
\begin{verbatim}
set x21=KGB(G,21)
support(x21)
Value: [1]
set q21=Parabolic:(support(x21),x21)
is_parabolic_theta_stable(q21)
Value: true
set L21=Levi(q21)
L21
Value: connected quasisplit real group with Lie algebra 'sl(2,R).u(1).u(1).u(1)'
\end{verbatim}

The $\theta$-stable parabolic subalgebra \texttt{q21} is the one defined by the KGB element \texttt{x21}, whose support consists of the \emph{second} simple root. Since \texttt{atlas} labels the simple roots of $\Delta^+(\frg, \frh_f)$ opposite to that of Knapp \cite{Kn}, the second one is the $\alpha_3$ in Figure~\ref{Fig-FI-Vogan}. Therefore, the Levi subgroup \texttt{L21} of \texttt{q21} has a simple factor $SL(2, \bbR)$.

Now let us set up an irreducible unitary representation and illustrate Theorem \ref{thm-Vogan-Paul-atlas}.

\begin{verbatim}
set p=parameter(KGB(G,21),[0,1,0,1],[-1/2,1,-1,0])
is_unitary(p)
Value: true
infinitesimal_character(p)
Value: [ 0, 1, 0, 1 ]/1
set (Q,q)=reduce_good_range(p)
q
Value: final parameter(x=2,lambda=[-3,2,-4,0]/2,nu=[-1,2,-2,0]/2)
rho_u(Q)
Value: [ 3, 0, 4, 2 ]/2
lambda(p)-rho_u(Q)=lambda(q)
Value: true
Q=q21
Value: true
\end{verbatim}
The last output says that the inducing module \texttt{q} does come from the $\theta$-stable parabolic subalgebra \texttt{q21}. To further identify \texttt{q}, let us start with the trivial representation of \texttt{L21}.
\begin{verbatim}
set t=trivial(L21)
t
Value: final parameter(x=2,lambda=[-1,2,-2,0]/2,nu=[-1,2,-2,0]/2)
goodness(t,G)
Value: "Good"
\end{verbatim}
Now let us move down certain digits of \texttt{t} outside the support of \texttt{L21} to get \texttt{q}.
\begin{verbatim}
set tm=parameter(x(t),lambda(t)-[1,0,1,0],nu(t))
tm=q
Value: true
goodness(q,G)
Value: "Weakly good"
\end{verbatim}
Therefore, the inducing module \texttt{q} is a weakly good unitary character of \texttt{L21}.
\begin{verbatim}
theta_induce_irreducible(q,G)
Value:
1*parameter(x=21,lambda=[0,1,0,1]/1,nu=[-1,2,-2,0]/2) [57]
\end{verbatim}
Thus doing cohomological induction from \texttt{q} recovers the original representation \texttt{p}.
\hfill\qed
\end{example}

Since development of the software \texttt{atlas} is still quickly ongoing, we sincerely recommend the very helpful weakly seminar \cite{AtV} to the reader.

\subsection{Scattered representations versus fully supported scattered representations}\label{sec-scattered-FS-scattered}

As defined in \cite{D20},  a \emph{scattered representation} is a member $\pi$ of $\widehat{G}^d$ which can \emph{not} be cohomologically induced from a member $\pi_L\in\widehat{L}^d$ from the \emph{good range}, where $L$ is the Levi subgroup of a proper $\theta$-stable parabolic subalgebra $\frq$ of $\frg$.

We refer to the fully supported members of $\widehat{G}^d$ as \emph{fully supported scattered representations} of $G$. Note that fully supported scattered representations must be scattered representations. Indeed, let $\pi$ be an arbitrary member of $\widehat{G}^d$ such that the KGB element $x$ of its \texttt{atlas} parameter $(x, \lambda, \nu)$ is fully supported, then $\frq(x)=\frg$. Since $\frq(x)$ is the minimum $\theta$-stable parabolic subalgebra from which $\pi$ can be cohomologically induced from the weakly good range, we conclude that $\pi$ must be scattered.

However, scattered representations need \emph{not} be fully supported. Perhaps the easiest example is the limit of holomorphic discrete series of $SL(2,\bbR)$, whose \texttt{atlas} parameter is as follows:
\begin{verbatim}
 final parameter(x=0,lambda=[0]/1,nu=[0]/1)
\end{verbatim}
It has zero infinitesimal character, and multiplicity-free $K$-types $1, 3, 5, 7, \dots$.
To make the story neat, we treat it as the starting point of the first row of Table \ref{table-SL2R}. We treat the limit of anti-holomorphic discrete series of $SL(2,\bbR)$ similarly, and summarize the results by saying that the Dirac series of $SL(2, \bbR)$ consists of $1$ fully supported scattered representations (the third row of Table \ref{table-SL2R}, the trivial representation) and $2$ strings of representations (the first row and the second row of Table \ref{table-SL2R}).

\begin{table}
\caption{Dirac series of $SL(2, \bbR)$, where $a$ runs over $\bbZ_{\geq 0}$}
\centering
\begin{tabular}{r|c|c|c|r}
$\# x$ &   $\lambda=\Lambda$   & $\nu$ &spin LKT=LKT  \\
\hline
$0$ & $[a]$ & $[0]$ & $[a+1]$\\
$1$ & $[a]$ & $[0]$ & $[-a-1]$\\
$2$ & $[1]$ & $[1]$ & $[0]$\\
\end{tabular}
\label{table-SL2R}
\end{table}

In general, let $N_{\rm FS}(G)$ denote the number of fully supported scattered representations of $G$.
For instance, we have $N_{\rm FS}(SL(2, \bbR))=1$.

Supported by the calculations in \cite{DDH,DDY}, we raise the following conjecture asserting that the reverse direction of Theorem \ref{thm-Vogan-coho-ind}(iv) should be true in certain special setting.

\begin{conj}\label{conj-unitarity-from-G-to-L-HP}
Let $\pi$ be any irreducible unitary $(\frg, K)$ module whose infinitesimal character $\Lambda$ meets the HP condition. Then $\pi_{L(x)}$ must be unitary.
\end{conj}

\section{The root system $\Delta(\frg, \frt_f)$}
\label{sec-coroots}

This section should be well-known to the experts. We learn the content from \cite{FV}, which might never be published. A good alternative reference is Steinberg \cite{St}.

We enumerate the simple roots for $\Delta^+(\frg, \frh_f)$ as follows:
\begin{align*}
&\alpha_1, \dots, \alpha_p, \mbox{(compact imaginary)}\\
&\beta_1, \dots, \beta_q, \mbox{(non-compact imaginary)}\\
&\gamma_1, \dots, \gamma_r, \theta(\gamma_1), \dots, \theta(\gamma_r), \mbox{(complex)}.
\end{align*}
Note that each part above may be absent.
We denote the corresponding fundamental weights by
$$
 \varpi(\alpha_1), \dots, \varpi(\alpha_p),  \varpi(\beta_1), \dots, \varpi(\beta_q),
 \varpi(\gamma_1), \dots, \varpi(\gamma_r),
 \varpi(\theta(\gamma_1)), \dots, \varpi(\theta(\gamma_r)).
$$
Let $\rho$ be the half sum of roots in $\Delta^+(\frg, \frh_f)$.
For each root $\alpha\in\Delta^+(\frg, \frh_f)$, we denote by $\overline{\alpha}$ its restriction to $\frt_f$, by $\alpha^\vee$ the coroot of $\alpha$. Note that $\theta(\gamma_j)^\vee=\theta(\gamma_j^\vee)$ for $1\leq j\leq r$.
We can label the simple roots $\gamma_1, \dots, \gamma_r$ so that $\gamma_j$, $\theta(\gamma_j)$, $\gamma_j^\vee$ and $\theta(\gamma_j^\vee)$  generate a subsystem of type $A_1\times A_1$ for $2\leq j\leq r$. However, when $j=1$ the subsystem can be of type $A_2$.

Collecting all these restricted roots $\overline{\alpha}$ for $\alpha\in\Delta(\frg, \frh_f)$, we get the  root system $\Delta(\frg, \frt_f)$ which may not be reduced.
Note that
\begin{equation}
\Delta_{\rm red}(\frg, \frt_f)=\{\overline{\alpha}\mid \overline{\alpha}/2 \notin \Delta(\frg, \frt_f)\}
\end{equation}
is a \emph{reduced} root system.

For any vector $\mu\in\frt_f^*$, we say that $\mu$ \emph{is integral for} $\Delta(\frg, \frt_f)$ if the pairing of $\mu$ with each coroot for $\Delta(\frg, \frt_f)$  is an integer.  Similarly, we say that $\mu$ \emph{is integral for} $\Delta(\frk, \frt_f)$ if the pairing of $\mu$ with each coroot for $\Delta(\frk, \frt_f)$ is an integer. It is obvious that if $\mu$ is integral for $\Delta(\frg, \frt_f)$, then $\mu$ must be integral for $\Delta(\frk, \frt_f)$.

Restricting all the roots of $\Delta^+(\frg, \frh_f)$ to $\frt_f$, we get a positive root system $\Delta^+(\frg, \frt_f)$. Its simple roots are
$$
\alpha_1, \dots, \alpha_p, \beta_1, \dots, \beta_q,
\overline{\gamma_1}, \dots, \overline{\gamma_r}.
$$
Here $\overline{\gamma_j}=\frac{\gamma_j+\theta(\gamma_j)}{2}$ for $1\leq j\leq r$.
Put
\begin{equation}\label{coroots-gt}
\alpha_1^\vee, \dots, \alpha_p^\vee, \beta_1^\vee, \dots, \beta_q^\vee,
\gamma_1^\vee+\theta(\gamma_1^\vee), \dots, \gamma_r^\vee+\theta(\gamma_r^\vee),
\end{equation}
and
\begin{equation}\label{fundwt-gt}
\varpi(\alpha_1), \dots, \varpi(\alpha_p), \varpi(\beta_1), \dots, \varpi(\beta_q), \frac{\varpi(\gamma_1)+\varpi(\theta(\gamma_1))}{2}, \dots, \frac{\varpi(\gamma_r)+\varpi(\theta(\gamma_r))}{2}.
\end{equation}
Note that $\mu$ is integral for $\Delta(\frg, \frt_f)$ if and only if the paring of $\mu$ with each coroot in \eqref{coroots-gt} is an integer, if and only if $\mu$ is an integer combination of \eqref{fundwt-gt}.

\begin{lemma}\label{lemma-linear}
Let $G$ be a simple linear real Lie group in the Harish-Chandra class. Let $\delta$ be the highest weight of any $K$-type. Then $\delta$ is integral for $\Delta(\frg, \frt_f)$.
\end{lemma}
\begin{proof}
We may and we will assume that $G$ is simply connected.
Let $X^*(H_f)$ be the lattice of rational characters of $H_f$. Define $X^*(T_f)$ similarly. Then as shown in \cite{FV},
$$
X^*(T_f)=X^*(H_f)/{\rm span}\, \{\lambda -\theta(\lambda)\mid \lambda \in X^*(H_f) \}.
$$
Since $G$ is simply-connected, the above denominator is
$$
{\rm span}\, \{\varpi(\gamma_i)-\varpi(\theta(\gamma_i))\mid 1\leq i\leq r\}.
$$
It follows  that \eqref{fundwt-gt} is a basis for $X^*(T_f)$.
\end{proof}

\begin{example}\label{exam-lemma-linear}
Consider the \emph{linear}  split real form $F_4$. This group is equal rank, i.e., $\frh_f=\frt_f$. It is connected but not simply connected. Indeed,
$$
K\cong {\rm Sp}(3)\times {\rm Sp}(1)/\{\pm 1\}.
$$
Let us adopt the Vogan diagram for its Lie algebra as in \cite{Kn}, see Figure \ref{Fig-FI-Vogan}.

\begin{figure}[H]
\centering
\scalebox{0.6}{\includegraphics{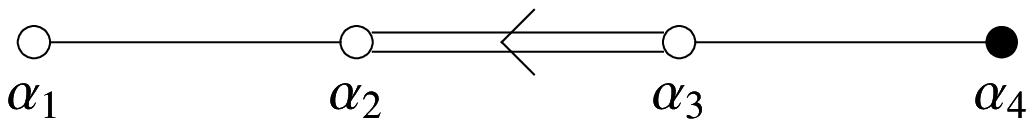}}
\caption{The Vogan diagram for FI}
\label{Fig-FI-Vogan}
\end{figure}

By choosing the Vogan diagram, we have actually fixed a positive root system $\Delta^+(\frg, \frt_f)$ with $\alpha_1, \dots, \alpha_4$ being the simple roots. Then correspondingly a positive root system $\Delta^+(\frk, \frt_f)$ is fixed, see Figure \ref{Fig-FI-K-Dynkin}, where the simple roots are
$$
\gamma_1=\alpha_1, \quad \gamma_2=\alpha_2, \quad \gamma_3=\alpha_3, \quad \gamma_4=2\alpha_1+4\alpha_2+3\alpha_3+2\alpha_4.
$$

\begin{figure}[H]
\centering
\scalebox{0.6}{\includegraphics{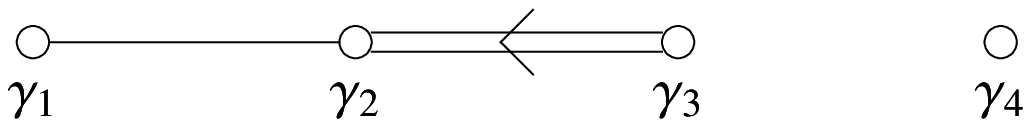}}
\caption{The Dynkin diagram for $\Delta^+(\frk, \frt_f)$}
\label{Fig-FI-K-Dynkin}
\end{figure}

Let us denote by $\xi_1, \xi_2, \xi_3, \xi_4$ (resp., $\varpi_1, \varpi_2, \varpi_3, \varpi_4$) be the fundamental weights for $\Delta^+(\frg, \frt_f)$ (resp., $\Delta^+(\frk, \frt_f)$). Let $a, b, c, d$ be arbitrary non-negative integers. Then one calculates that
\begin{equation}\label{FI-Kg}
a\varpi_1 + b \varpi_2 + c\varpi_3 + d \varpi_4=a\xi_1 + b\xi_2+ c\xi_3 +(-\frac{a}{2} - b - \frac{3}{2}c + \frac{d}{2}) \xi_4.
\end{equation}
Since $a\varpi_1 + b \varpi_2 + c\varpi_3 + d \varpi_4$ is the highest weight of a $K$-type if and only if $a+c+d$ is even, one sees from \eqref{FI-Kg} that Lemma \ref{lemma-linear} holds.

However, if we pass to the universal covering group of the linear real split $F_4$, which is \emph{non-linear}, then \eqref{FI-Kg} says that Lemma \ref{lemma-linear} fails.\hfill\qed
\end{example}

\section{Proof of the non-vanishing criterion}

We collect elements of $W(\frg, \frt_f)$ moving the dominant Weyl chamber for $\Delta^+(\frg, \frt_f)$ within the dominant Weyl chamber for $\Delta^+(\frk, \frt_f)$ as $W(\frg, \frt_f)^1$. It turns out that the multiplication map
$$
W(\frk, \frt_f) \times W(\frg, \frt_f)^1 \to W(\frg, \frt_f)
$$
induces a bijection \cite{Ko}. Therefore, the set $W(\frg, \frt_f)^1$ has cardinality $s$ defined in \eqref{s}.  Let us enumerate its elements as
\begin{equation}\label{WgtOne}
W(\frg, \frt_f)^1=\{w^{(0)}=e, w^{(1)}, \dots, w^{(s-1)} \}.
\end{equation}
Recall that the highest weights of ${\rm Spin}_G$ as $\frk$-module are exactly
\begin{equation}\label{Rhonj}
\rho_n^{(j)}=w^{(j)}\rho-\rho_K, \quad 0\leq j \leq s-1.
\end{equation}

\noindent\emph{Proof of Theorem \ref{thm-main}.}
As mentioned at the end of Section \ref{sec-pre}, it suffices to prove the ``$\Leftarrow$" direction. Assume that $H_D(Z)\neq 0$. Take any $\widetilde{L\cap K}$-type $\gamma_L$ in $H_D(Z)$. Here $\widetilde{L\cap K}$ stands for the pin covering group of $L\cap K$.
By Theorem \ref{thm-HP}, there exists $w_1\in W(\frl, \frt_f)$ such that
\begin{equation}\label{wlambdaL}
\gamma_L+\rho_{L\cap K}=w_1\lambda_L.
\end{equation}
In particular, it follows that $w_1 \lambda_L$ is dominant integral regular for $\Delta^+(\frl\cap\frk, \frt_f)$.

Put $\gamma_G:=\gamma_L+\rho(\fru\cap\frp)$. Then due to \eqref{rhoukp} and that
$w_1 \Delta(\fru)=\Delta(\fru)$, we have
\begin{equation}\label{inf-char-K-tilde-type}
\gamma_G+\rho_K=\gamma_L+\rho_{L\cap K}+\rho(\fru)=w_1 \lambda_L + \rho(\fru)=w_1(\lambda_L +\rho(\fru)).
\end{equation}
We \emph{claim} that $\gamma_G+\rho_K$ is dominant integral regular for $\Delta^+(\frk, \frt_f)$.

Firstly, since $\lambda_L+\rho(\fru)\in\frt_f^*$ is assumed further to meet the HP condition, there exist $w\in W(\frg, \frt_f)^1$ and $0\leq j\leq s-1$ such that
\begin{equation}\label{HP-G}
\{\delta-\rho_n^{(j)} \} +\rho_K=w(\lambda_L+\rho(\fru)).
\end{equation}
Note that
\begin{equation}\label{Ktilde-inf-char}
\{\delta-\rho_n^{(j)} \} +\rho_K=\delta-\rho_n^{(j)}+\sum_{n_i}n_i\gamma_i+\rho_K=
\delta-w^{(j)}\rho+2\rho_K + \sum_{n_i}n_i\gamma_i,
\end{equation}
where $n_i$ are some non-negative integers, and $\gamma_i$ are roots in $\Delta^+(\frk, \frt_f)$.
It follows from Lemma \ref{lemma-linear} and \eqref{Ktilde-inf-char} that $\{\delta-\rho_n^{(j)} \} +\rho_K$ is integral for $\Delta(\frg, \frt_f)$. Since $W(\frl, \frt_f)\leq W(\frg, \frt_f)$,  we conclude from \eqref{HP-G} and \eqref{wlambdaL}  that $w_1\lambda_L+\rho(\fru)$ is integral for $\Delta(\frg, \frt_f)$ as well. In particular,
$w_1\lambda_L+\rho(\fru)$ is integral for $\Delta(\frk, \frt_f)$.

Secondly, since $\lambda_L+\rho(\fru)$ is assumed to be good, one sees that
$$
\langle w_1\lambda_L+\rho(\fru), \alpha^{\vee}\rangle=\langle \lambda_L+\rho(\fru), w_1^{-1}(\alpha)^{\vee}\rangle >0,
\quad \forall \alpha\in\Delta^+(\frk\cap\fru, \frt_f).
$$
Moreover, for any $\alpha\in\Delta^+(\frl\cap\frk, \frt_f)$.
$$
\langle w_1\lambda_L+\rho(\fru), \alpha^{\vee}\rangle=\langle \gamma_L+\rho_{L\cap K}+\rho(\fru), \alpha^{\vee}\rangle=\langle \gamma_L+\rho_{L\cap K}, \alpha^{\vee}\rangle >0.
$$
Therefore, $w_1\lambda_L+\rho(\fru)$ is dominant regular for $\Delta^+(\frk, \frt_f)$.

To sum up, $w_1\lambda_L+\rho(\fru)$ is dominant integral regular for $\Delta^+(\frk, \frt_f)$. Thus the claim holds. Hence $\gamma_G=w_1\lambda_L+\rho(\fru)-\rho_K$ is  dominant integral for $\Delta^+(\frk, \frt_f)$. Thus $\gamma_G=\gamma_L+\rho(\fru\cap\frp)$ occurs in $H_D(\caL_S(Z))$ by Theorem \ref{thm-DH}.
\hfill\qed

\begin{rmk}\label{rmk-integral-inf-char}
Let $G$ be a simple linear Lie group in the Harish-Chandra class. Assume that $\Lambda\in\frh_f^*$ satisfies the HP condition. Then the above proof says that $\Lambda$ is conjugate to a vector in $\frt_f^*$ which is a non-negative integer combination of \eqref{fundwt-gt} under the action of $W(\frg, \frh_f)$.
\end{rmk}

Finally, let us present an example showing that Theorem \ref{thm-main} does not hold if the good range condition is loosen to be weakly good.

\begin{example}\label{exam-weakly-good}
Consider the representation \texttt{p} of \texttt{F4\_s} studied in Example \ref{exam-atlas-basic}.  As mentioned earlier, this is a weakly good $A_{\frq}(\lambda)$ module. The bottom layer of $\texttt{p}$ consists of the unique $K$-type
$$
\lambda + 2\rho(\fru\cap\frp)=\varpi_2+8\varpi_4,
$$
which is also the unique lowest $K$-type of $\texttt{p}$.
This $K$-type has spin norm $\sqrt{15}$, while $\|\Lambda\|=\sqrt{11}$. Thus this unique bottom layer $K$-type can not contribute to $H_D(\texttt{p})$, which then must vanish by Proposition 4.5 of \cite{DH}.
\hfill\qed
\end{example}

\section{The number of strings in $\widehat{G}^d$}\label{sec-number-strings}

As demonstrated in Section \ref{sec-scattered-FS-scattered} and Example 4.2 of \cite{DDY}, we can use translation functor to merge any Dirac series representation which is not fully supported into a string. See Example \ref{exam-EII-string-cancellation} as well. This will allow us to present the Dirac series neatly, and in particular, allows us to count the number of strings in $\widehat{G}^d$. In other words, \texttt{atlas} teaches us that it is quite natural to arrange the Dirac series of $G$ according to the support of their KGB elements.

For simplicity, we assume that $G$ is equal rank. Let $\{\xi_1, \dots, \xi_l\}$ be the corresponding fundamental weights corresponding to the simple roots of $\Delta^+(\frg, \frh_f)$.

We assume that Conjecture \ref{conj-unitarity-from-G-to-L-HP} holds for $G$. We further assume that the following \emph{binary condition} holds for $G$: Let $\pi$ with final \texttt{atlas} parameter $p=(x, \lambda, \nu)$ be an irreducible unitary representation. Let $\Lambda=\sum_{i=1}^{l} n_i \zeta_i$ be the infinitesimal character of $\pi$ which is integral (see Remark \ref{rmk-integral-inf-char}). Then each $n_i$ corresponding to a simple root in \texttt{support(x)} is either $0$ or $1$. The binary condition should be closely related to Conjecture 5.7' of \cite{SV}.

Let $S$ be any \emph{proper} subset of the simple roots of $\Delta^+(\frg, \frh_f)$.
We collect the dominant integral HP infinitesimal characters $\Lambda$ whose coordinates are $0$ or $1$ on the digits corresponding to $S$, and whose coordinates outside $S$ are $1$ by $\Omega(S)$. Denote by $N(S)$ the number of Dirac series representations with infinitesimal character in $\Omega(S)$ and support $S$.
Put
$$
N_i=\sum_{\#S=i} N(S).
$$
Then $N_0+N_1+\cdots+N_{l-1}$ is the number of strings in $\widehat{G}^d$.

\section{Dirac series of $E_{6(2)}$}

In this section, we fix $G$ as the simple real exceptional linear Lie group \texttt{E6\_q} in \texttt{atlas}. This connected group is equal rank. That is, $\frh_f=\frt_f$. It has center $\bbZ/3\bbZ$. The Lie algebra $\frg_0$ of $G$ is denoted as \texttt{EII} in \cite[Appendix C]{Kn}.
Note that
$$
-\dim \frk +\dim \frp=-38+40=2.
$$
Therefore, the group $G$ is also called $E_{6(2)}$ in the literature.

We present a Vogan diagram for $\frg_0$ in Fig.~\ref{Fig-EII-Vogan}, where $\alpha_1=\frac{1}{2}(1, -1,-1,-1,-1,-1,-1,1)$, $\alpha_2=e_1+e_2$ and $\alpha_i=e_{i-1}-e_{i-2}$ for $3\leq i\leq 6$. By specifying a Vogan diagram, we have actually fixed a choice of positive roots $\Delta^+(\frg, \frt_f)$.  Let $\zeta_1, \dots, \zeta_6\in\frt_f^*$ be the corresponding fundamental weights for $\Delta^+(\frg, \frt_f)$. The dual space $\frt_f^*$ will be identified with $\frt_f$ under the Killing form $B(\cdot, \cdot)$.
We will use $\{\zeta_1, \dots, \zeta_6\}$ as a basis to express the \texttt{atlas} parameters $\lambda$, $\nu$ and the infinitesimal character $\Lambda$. More precisely, in such cases, $[a, b, c, d, e, f]$ will stand for the vector $a\zeta_1+\cdots+f \zeta_6$.

\begin{figure}[H]
\centering
\scalebox{0.6}{\includegraphics{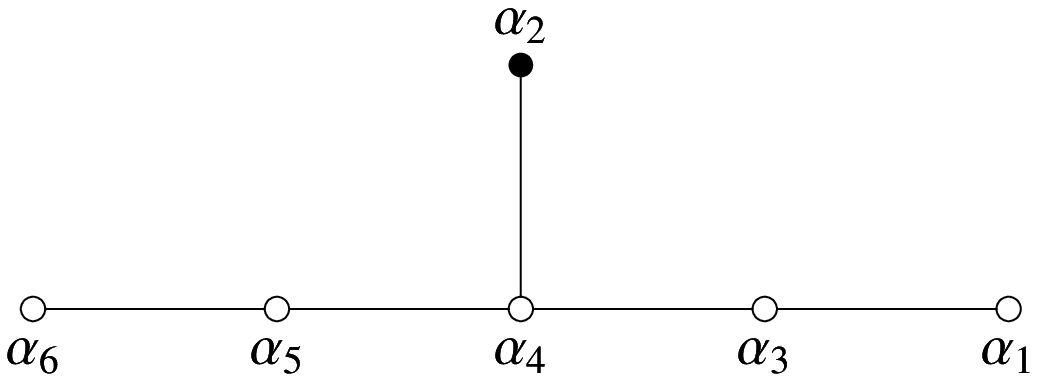}}
\caption{The Vogan diagram for EII}
\label{Fig-EII-Vogan}
\end{figure}

Put $\gamma_i=\alpha_{7-i}$ for $1\leq i\leq 4$, $\gamma_5=\alpha_1$, and
$$
\gamma_6=\alpha_1+ 2\alpha_2+ 2\alpha_3+ 3\alpha_4+ 2\alpha_5 +\alpha_6=
(\frac{1}{2},\frac{1}{2},\frac{1}{2},\frac{1}{2},
\frac{1}{2},-\frac{1}{2},-\frac{1}{2},\frac{1}{2}),
$$
which is the highest root of $\Delta^+(\frg, \frt_f)$.
Then $\gamma_1, \dots, \gamma_6$ are the simple roots of $\Delta^+(\frk, \frt_f)=\Delta(\frk, \frt_f)\cap \Delta^+(\frg, \frt_f)$. We present the Dynkin diagram of $\Delta^+(\frk, \frt_f)$  in Fig.~\ref{Fig-EII-K-Dynkin}.
Let $\varpi_1, \dots, \varpi_6\in \frt_f^*$ be the corresponding fundamental weights.

\begin{figure}[H]
\centering
\scalebox{0.6}{\includegraphics{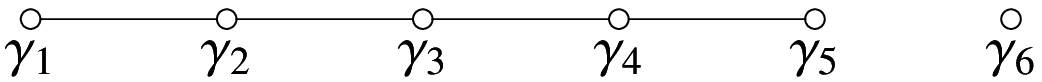}}
\caption{The Dynkin diagram for $\Delta^+(\frk, \frt_f)$}
\label{Fig-EII-K-Dynkin}
\end{figure}

Let $E_{\mu}$ be the $\frk$-type with highest weight $\mu$.
We will use $\{\varpi_1, \dots, \varpi_6\}$ as a basis to express $\mu$.
Namely, in such a case, $[a, b, c, d, e, f]$ stands for the vector $a \varpi_1+b \varpi_2 + c \varpi_3+d \varpi_4 +e \varpi_5+ f\varpi_6$. For instance,
$$
\beta:=\alpha_1+ \alpha_2+ 2\alpha_3+ 3\alpha_4+ 2\alpha_5 +\alpha_6=[0,0,1,0,0,1]
$$ and $\dim \frp=\dim E_{\beta}=40$. The $\frk$-type $E_{[a, b, c, d, e, f]}$ has lowest weight $[-e, -d, -c, -b, -a, -f]$. Therefore, $E_{[e, d, c, b, a, f]}$ is the contragredient $\frk$-type of $E_{[a, b, c, d, e, f]}$. For $a, b,c ,d ,e, f\in\bbZ_{\geq 0}$, we have that $E_{[a, b, c, d, e, f]}$ is  a $K$-type if and only if
\begin{equation}\label{EII-K-type}
a+c+e+f \mbox{ is even.}
\end{equation}
Note that
$$
\# W(\frg, \frt_f)^1=\frac{51840}{1440}=36.
$$

In the current setting, the spin norm of the $\frk$-type $E_{\mu}$ specializes as
$$
\|\mu\|_{\rm spin} =\min_{0\leq j \leq 35}\|\{\mu-\rho_n^{(j)}\}+\rho_K\|.
$$
Note that $E_{\{\mu-\rho_n^{(j)}\}}$ is the PRV component \cite{PRV} of the tensor product of $E_{\mu}$ with the contragredient $\frk$-type of $E_{\rho_n^{(j)}}$.
Let $\pi$ be any infinite-dimensional irreducible $(\frg, K)$ module with infinitesimal character $\Lambda$. Put the spin norm of $\pi$ as
$$
\|\pi\|_{\rm spin}=\min \|\delta\|_{\rm spin},
$$
where $V_{\delta}$ runs over all the $K$-types of $\pi$.
If $V_{\delta}$ attains $\|\pi\|_{\rm spin}$, we will call it a \emph{spin-lowest $K$-type} of $\pi$. If $\pi$ is further assumed to be unitary,
Parthasarathy's Dirac operator inequality \cite{Pa2} can be rephrased as
\begin{equation}\label{Dirac-inequality}
\|\pi\|_{\rm spin}\geq \|\Lambda\|.
\end{equation}
Moreover, as shown in \cite{D13}, the equality happens in \eqref{Dirac-inequality} if and only if $\pi$ has non-zero Dirac cohomology, and in this case, it is exactly the (finitely many) spin-lowest $K$-types of $\pi$ that contribute to $H_D(\pi)$.

Since \texttt{E6\_q} is not of Hermitian symmetric type, results from \cite{Vog80} says that the $K$-type $V_{\delta + n\beta}$ must show up in $\pi$ for any non-negative integer $n$ provided that $V_{\delta}$ occurs in $\pi$. We call them the \emph{Vogan pencil} starting from $V_{\delta}$.   Now it follows from \eqref{Dirac-inequality} that
\begin{equation}\label{Dirac-inequality-improved}
\|\delta+n\beta\|_{\rm spin}\geq \|\Lambda\|, \quad \forall n\in \bbZ_{\geq 0}.
\end{equation}
In other words, whenever \eqref{Dirac-inequality-improved} fails, we can conclude that $\pi$ is non-unitary. Distribution of spin norm along Vogan pencils has been discussed in Theorem C of \cite{D20}. In practice,
we will take $\delta$ to be a lowest $K$-type of $\pi$ and use the corresponding Vogan pencil to do non-unitarity test.

This section aims to classify the Dirac series of $E_{6(2)}$, see Theorem \ref{thm-EII}.

\subsection{Fully supported scattered representations of $E_{6(2)}$}\label{sec-FS-EII}

This subsection aims to seive out all the fully supported scattered Dirac series representations for $E_{6(2)}$ using the algorithm in \cite{D20}.

\begin{lemma}\label{lemma-E62-HP}
Let $\Lambda=a\zeta_1+b\zeta_2+c\zeta_3+d\zeta_4+e\zeta_5+f \zeta_6$ be the infinitesimal character of any Dirac series representation $\pi$ of $E_{6(2)}$ which is dominant with respect to $\Delta^+(\frg, \frt_f)$. Then $a$, $b$, $c$, $d$, $e$, $f$ must be non-negative integers such that $a+c>0$, $b+d>0$, $c+d>0$, $d+e>0$ and $e+f>0$.
\end{lemma}
\begin{proof} It follows from Remark \ref{rmk-integral-inf-char} that $a$, $b$, $c$, $d$, $e$, $f$ must be non-negative integers.

Now if $a+c=0$, i.e., $a=c=0$, then a direct check says that for any $w\in W(\frg, \frt_f)^1$, at least one coordinate of $w\Lambda$ in terms of the basis $\{\varpi_1, \dots, \varpi_6\}$ vanishes. Therefore,
$$
\{\mu-\rho_n^{(j)}\} + \rho_K =w \Lambda
$$
could not hold. This proves that $a+c>0$, other inequalities can be similarly obtained.
\end{proof}

To obtain all the fully supported scattered Dirac series representations of $E_{6(2)}$, now it suffices to consider all the infinitesimal characters $\Lambda=[a, b, c, d, e, f]$ such that
\begin{itemize}
\item[$\bullet$] $a$, $b$, $c$, $d$, $e$, $f$ are non-negative integers;
\item[$\bullet$] $a+c>0$, $b+d>0$, $c+d>0$, $d+e>0$, $e+f>0$;
\item[$\bullet$] $\min\{a, b, c, d, e, f\}=0$;
\item[$\bullet$] there exists a fully supported KGB element $x$ such that $\|\Lambda-\theta_x\Lambda\|\leq \|2\rho\|$.
\end{itemize}
Let us collect them as $\Phi$. It turns out that $\Phi$ has cardinality $58061$. There are $21$ elements of $\Phi$ whose largest entry equals to $1$:
\begin{align*}
&[0, 0, 1, 1, 0, 1], [0, 0, 1, 1, 1, 0], [0, 0, 1, 1, 1, 1], [0, 1, 1,
0, 1, 0], [0, 1, 1, 0, 1, 1], [0, 1, 1, 1, 0, 1], \\
&[0, 1, 1, 1, 1, 0], [0, 1, 1, 1, 1, 1], [1, 0, 0, 1, 0, 1], [1, 0, 0, 1, 1, 0], [1, 0, 0,
1, 1, 1], [1, 0, 1, 1, 0, 1],\\
&[1, 0, 1, 1, 1, 0], [1, 0, 1, 1, 1, 1], [1, 1, 0, 1, 0, 1], [1, 1, 0, 1, 1, 0], [1, 1, 0, 1, 1, 1], [1, 1, 1, 0, 1, 0],\\
&[1, 1, 1, 0, 1, 1], [1, 1, 1, 1, 0, 1], [1, 1, 1, 1, 1, 0].
\end{align*}
A careful study of the irreducible unitary representations under the above $21$ infinitesimal characters leads us to Section \ref{sec-appendix}. Let $p=(x, \lambda, \nu)$ be a fully supported scattered Dirac series representation of $E_{6(2)}$ with infinitesimal character $\Lambda$. It can happen that there exists another fully supported scattered Dirac series representation  $p'=(x', \lambda', \nu')$ of $E_{6(2)}$ with infinitesimal character $\Lambda'$ such  that $\Lambda'$ (resp. $\lambda'$, $\nu'$) is obtained from $\Lambda$ (resp. $\lambda$, $\nu$) by exchanging its first and sixth, third and fifth coordinates. Moreover, the spin lowest $K$-types of $p'$ are exactly the contragredient $K$-types of those of $p$.  Whenever this happens, we will \emph{fold} $p'$ by omitting $\lambda'$, $\nu'$ and the spin-lowest $K$-types of $p'$. Instead, we only present $x'$ in the \textbf{bolded} fashion in the last column along with $p$. The reader can recover $p'$ easily. For instance,  let us come to the following representation in Table \ref{table-EII-111011},
$$p=\rm{parameter(KGB(G,1624)}, [1,1,4,-1,1,1], [1,1,4,-3,1,1]),$$
which has spin-lowest $K$-types $[0,3,0,0,0,0]$, $[0,3,1,0,0,1]$, $[0,3,2,0,0,2]$.
The bolded \textbf{1623} says that
$$p'=\rm{parameter(KGB(G,1623)}, [1,1,1,-1,4,1], [1,1,1,-3,4,1])$$
is also a fully supported scattered Dirac series representation. Moreover, its spin-lowest $K$-types are
$[0,0,0,3,0,0]$, $[0,0,1,3,0,1]$, $[0,0,2,3,0,2]$.

For the other $58040$ elements of $\Phi$, we use Parthasarathy's Dirac operator inequality, and distribution of spin norm along the Vogan pencil starting from one lowest $K$-type to verify that there is no fully supported irreducible unitary representations with infinitesimal character $\Lambda$. This method turns out to be very effective in non-unitarity test. Indeed, it fails only on the following infinitesimal characters of $\Phi$:
\begin{align*}
[0, 0, 1, 1, 0, 2], [0, 2, 1, 0, 1, 0], [0, 2, 1, 0, 1, 1], [1, 0, 0, 1, 0, 2],\\
[1, 2, 1, 0, 1, 0], [2, 0, 0, 1, 0, 1], [2, 0, 0, 1, 0, 2], [2, 0, 0, 1, 1, 0].
\end{align*}
However, a more careful look says that there is no fully supported irreducible unitary representation under them. Let us provide one example.

\begin{example}
Consider the infinitesimal character
$\Lambda=[0, 0, 1, 1, 0, 2]$ for $\texttt{E6\_q}$.

\begin{verbatim}
G:E6_q
set all=all_parameters_gamma(G,[0,0,1,1,0,2])
#all
Value: 263
set i=0
void: for p in all do if #support(x(p))=6 then i:=i+1 fi od
i
Value: 110
\end{verbatim}
There are $263$ irreducible representations with infinitesimal character $\Lambda$, among which $110$ are fully supported. A careful look at them says that the non-unitarity test using the pencil starting from one of the lowest $K$-types fails only for the following representation:
\begin{verbatim}
set p=parameter(KGB(G,1536),[0,0,3,0,1,1],[0,0,4,-2,0,2])
\end{verbatim}
Indeed, it has a unique lowest $K$-type $\delta=[0,3,0,0,0,0]$, and the minimum spin norm along the pencil $\{\delta+n\beta\mid n\in\bbZ_{\geq 0}\}$ is $\sqrt{42}$, while $\|\Lambda\|=6$. Thus Dirac inequality does not work here. Instead, we check the unitarity of \texttt{p} directly:
\begin{verbatim}
is_unitary(p)
Value: false
\end{verbatim}
Thus there is \emph{no} fully supported irreducible unitary representation with infinitesimal character $\Lambda$.
\hfill\qed
\end{example}

\subsection{Conjecture \ref{conj-unitarity-from-G-to-L-HP} and binary condition for $E_{6(2)}$}\label{sec-EII-esc}

In this subsection, let us verify that $E_{6(2)}$ satisfies Conjecture \ref{conj-unitarity-from-G-to-L-HP} and the binary condition.
Let $p=(x, \lambda, \nu)$ be any irreducible unitary representation with infinitesimal character $\Lambda=[a, b, c, d, e, f]$ meeting the requirements in Lemma \ref{lemma-E62-HP}. It suffices to check that the coordinates of $\Lambda$ within the support of $x$ are either $0$ or $1$,  and that the inducing module $p_{L(x)}$ is unitary.

\begin{example}
Consider the case that $\texttt{support(x)=[0, 1, 2, 3, 4]}$. There are $168$ such KGB elements in total. We compute that there are $24109$ infinitesimal characters $\Lambda=[a, b, c, d, e, f]$ in total such that
\begin{itemize}
\item[$\bullet$] $a$, $b$, $c$, $d$, $e$ are non-negative integers, $f=0$ or $1$;
\item[$\bullet$] $a+c>0$, $b+d>0$, $c+d>0$, $d+e>0$, $e+f>0$;
\item[$\bullet$] $\min\{a, b, c, d, e, f\}=0$;
\item[$\bullet$] there exists a  KGB element $x$ with support $[0, 1, 2, 3, 4]$ such that $\|\Lambda-\theta_x\Lambda\|\leq \|2\rho\|$.
\end{itemize}
The reason that we can reduce the verification of the unitarity of $p_{L(x)}$ to the cases $f=0$ or $1$ is due to the relation between translation functor and cohomological induction, see Theorem 7.237 of \cite{KV}.

As in Section \ref{sec-FS-EII}, we exhaust all the irreducible unitary representations under these infinitesimal characters  with the above $168$ KGB elements. It turns out that such representations occur only when $a, b, c, d, e=0$ or $1$.  Then we check that each $p_{L(x)}$ is indeed unitary. \hfill\qed
\end{example}

All the other non fully supported KGB elements are handled similarly. Eventually we conclude that $E_{6(2)}$ satisfies Conjecture \ref{conj-unitarity-from-G-to-L-HP} and the binary condition.

\subsection{Number of strings in $\widehat{E_{6(2)}}^d$}

Thanks to Section \ref{sec-EII-esc}, each representation in $\widehat{E_{6(2)}}^d$ whose KGB element is not fully supported can be equipped into a string in the fashion of Example \ref{exam-EII-string-cancellation}. Using the formula in Section \ref{sec-number-strings}, let us pin down the number of strings in $\widehat{E_{6(2)}}^d$ in this subsection.

We compute that
$$
N([0,1,2,3,4])=N([1,2,3,4,5])=45, \quad N([0,2,3,4,5])=29,
$$
and that
$$
N([0,1,3,4,5])=N([0,1,2,3,5])=7, \quad N([0,1,2,4,5])=1.
$$
In particular, it follows that
$$
N_5=2\times 45 + 29+ 2\times 7 + 1=134.
$$
We also compute that
$$
N_0=36, \quad N_1=60, \quad N_2=80, \quad N_3=115, \quad N_4=151.
$$
Therefore, the total number of strings for $E_{6(2)}$ is equal to
$$
\sum_{i=0}^{5} N_i=576.
$$

To end up with this section, we mention that some auxiliary files have been built up to facilitate the classification of the Dirac series of $E_{6(2)}$. They are available via the following link:
\begin{verbatim}
https://www.researchgate.net/publication/353352799_EII-Files
\end{verbatim}

\section{Cancellation in Dirac cohomology}\label{sec-can-EII}

It is interesting to note that cancellation continues to happen within the Dirac cohomology of some fully supported scattered members of $\widehat{E_{6(2)}}^{\rm d}$ when passing  to Dirac index.  There are $10$ such representations in total, and in each case the Dirac index vanishes. We will mark their KGB elements with stars (see Section \ref{sec-appendix}).

\begin{example}\label{exam-EII-cancellation}
Consider the following representation $\pi$
\begin{verbatim}
final parameter(x=1649,lambda=[2,-2,0,4,-1,1]/1,nu=[2,-3,0,5,-3,2]/2)
\end{verbatim}
It has infinitesimal character $[1, 0, 0, 1, 0, 1]$, which is conjugate to $\rho_K$ under the action of $W(\frg, \frt_f)$. The representation $\pi$ has four spin LKTs:
$$
[2,0,2,1,0,2]=\rho_n^{(25)}, \, [2,1,1,0,1,4]=\rho_n^{(13)}, \, [3,0,1,1,1,1]=\rho_n^{(27)}, \, [3,1,0,0,2,3]=\rho_n^{(16)}.
$$
Therefore, $H_D(\pi)$ consists of four copies of the trivial $\widetilde{K}$-type.

Note that $-\rho_n^{(25)}$, $-\rho_n^{(13)}$, $-\rho_n^{(27)}$, $-\rho_n^{(16)}$ are the lowest weights of $E_{\rho_n^{(22)}}$, $E_{\rho_n^{(15)}}$, $E_{\rho_n^{(26)}}$ and $E_{\rho_n^{(18)}}$, respectively.
Moreover,
$w^{(22)}=s_2s_4s_5s_6s_3s_4s_5s_1$ and $w^{(15)}=s_2s_4s_5s_6s_3s_4$
have even lengths, while
$w^{(26)}=s_2s_4s_5s_6s_3s_4s_5s_2s_1$ and $w^{(18)}=s_2s_4s_5s_6s_3s_4s_2$
have odd lengths.
Therefore, two trivial $\widetilde{K}$-type live in the even part of $H_D(\pi)$,
while the other two live in the odd part of $H_D(\pi)$. See Lemma 2.3 of \cite{DW21}. As a consequence, the Dirac index of $\pi$ vanishes.

Note that ${\rm DI}(\pi)$ can also be easily calculated by \texttt{atlas} using \cite{MPVZ}:
\begin{verbatim}
G:E6_q
set p=parameter (KGB (G)[1649],[3,-1,0,4,-3,3]/1,[2,-3,0,5,-3,2]/2)
show_dirac_index(p)
Dirac index is 0
\end{verbatim}
which agrees with the previous calculation. \hfill\qed
\end{example}

\begin{example}\label{exam-EII-minimal}
The first entry of Table \ref{table-EII-111011} is the minimal representation of $E_{6(2)}$. It has one LKT $[0,0,0,0,0,2]$ and four spin LKTs
$$
[0,0,1,0,0,3], \, [0,0,2,0,0,4], \, [0,0,3,0,0,5], \, [0,0,4,0,0,6].
$$
We calculate that its Dirac index vanishes.
 \hfill\qed
\end{example}

\begin{example}\label{exam-EII-string-cancellation}
Consider the following irreducible representation $\pi_{0,0}$:
\begin{verbatim}
set p00=parameter(KGB(G)[851],[-1,1,1,1,1,-1]/1,[-2,1,1,0,1,-2]/1)
is_unitary(p00)
Value: true
infinitesimal_character(p00)
Value: [ 0, 1, 1, 0, 1, 0 ]/1
support(KGB(G,851))
Value: [1,2,3,4]
show_dirac_index(p00)
Dirac index is 0
\end{verbatim}
The representation $\pi_{0,0}$ has two multiplicity-free spin LKTs: $[1,1,1,1,1,1]$ and $[2,0,2,0,2,2]$. Cancellation happens within $H_D(\pi_{0,0})$, resulting that ${\rm DI}(\pi_{0,0})=0$.
\begin{verbatim}
set (Q,q00)=reduce_good_range(p00)
goodness(q00,G)
Value: "Weakly good"
Levi(Q)
Value: connected quasisplit real group with Lie algebra 'so(4,4).u(1).u(1)'
is_unitary(q00)
Value: true
\end{verbatim}
Thus the representation $\pi_{0,0}$ is cohomologically induced from the irreducible unitary representation \texttt{q00} of  \texttt{Levi(Q)}. Note that \texttt{q00} is actually the \emph{minimal} representation of  \texttt{Levi(Q)}. Moreover, we compute that cancellation happens within the Dirac cohomology of \texttt{q00} and its Dirac index vanishes.

Now for any non-negative integers $a, f$, we  move  the first and last coordinates of $\lambda(\texttt{q00})$ to $a$ and $f$, respectively. Then we  arrive at the irreducible unitary representation \texttt{qaf} of \texttt{Levi(Q)}. Doing cohomological induction from \texttt{qaf} will give us an irreducible unitary representation $\pi_{a, f}$ of $G$. By Theorem \ref{thm-DH}, $\pi_{a, f}$ must have non-zero Dirac cohomology. Indeed, $\pi_{a, f}$ has two multiplicity-free spin LKTs: $[1,a+1,1,f+1,1,1]$ and $[2,a,2,f,2,2]$. Moreover, by Proposition 4.1 of \cite{DW21-1}, we have that ${\rm DI}(\pi_{a, f})=0$.

Although it is not  easy to check that whether $\pi_{0, 0}$ is a scattered member of $\widehat{E_{6(2)}}^d$ or not, we can simply embed it into the string $\pi_{a, f}$, where $a, f$ run over $\bbZ_{\geq 0}$. \hfill\qed
\end{example}
\begin{rmk}\label{rmk-cancellation-quaternionic}
The Levi subgroup \texttt{Levi(Q)} is also of quaternionic type.
The cancellation phenomenon seems to be closely related to quaternionic real forms.
\end{rmk}

\section{Special unipotent representations}\label{sec-EII-su}

In the list \cite{LSU} offered by Adams, the group \texttt{E6\_q} has $47$  special unipotent representations. Ten of them are also fully supported with non-zero Dirac cohomology. We mark the nine non-trivial ones with the subscript $\clubsuit$ in Section \ref{sec-appendix}. It is interesting to note that except for the first entry of Table \ref{table-EII-111011} (the minimal representation), the other eight non-trivial special unipotent representations are all $A_{\frq}(\lambda)$ modules. A brief summary is given below.

\begin{center}
\begin{tabular}{c|c}
special unipotent representation & realization as an $A_q(\lambda)$ module \\
\hline
1st entry of Table \ref{table-EII-110101} & \texttt{KGP(G,[0,1,2,3,4])[1]}, $f\downarrow 6$, Weakly fair\\
\hline
2nd entry of Table \ref{table-EII-110101} & \texttt{KGP(G,[0,1,2,3,4])[0]}, $f\downarrow 6$, Weakly fair\\
\hline
1st entry of Table \ref{table-EII-100101} & \texttt{KGP(G,[0,2,3,4,5])[1]}, $b\downarrow 6$, None\\
\hline
2nd entry of Table \ref{table-EII-100101} & \texttt{KGP(G,[0,2,3,4,5])[2]}, $b\downarrow 6$, None\\
\hline
3rd entry of Table \ref{table-EII-100101} & \texttt{KGP(G,[0,2,3,4,5])[2]}, $b\downarrow 5$, Fair\\
\hline
4th entry of Table \ref{table-EII-100101} & \texttt{KGP(G,[0,2,3,4,5])[1]}, $b\downarrow 5$, Fair\\
\hline
6th entry of Table \ref{table-EII-100101} & \texttt{KGP(G,[0,2,3,4,5])[0]}, $b\downarrow 6$, None\\
\hline
9th entry of Table \ref{table-EII-100101} & \texttt{KGP(G,[0,2,3,4,5])[0]}, $b\downarrow 5$, Fair\\
\hline
\end{tabular}
\end{center}

\begin{example}\label{EII-su-Aq}
Let us explain the second row of the above table.
\begin{verbatim}
G:E6_q
set P=KGP(G,[0,1,2,3,4])
#P
Value: 4
void: for i:4 do prints(P[i],"  ",is_parabolic_theta_stable(P[i])) od
([0,1,2,3,4],KGB element #1260)  true
([0,1,2,3,4],KGB element #1434)  true
([0,1,2,3,4],KGB element #1776)  false
([0,1,2,3,4],KGB element #1790)  false
set L=Levi(P[0])
set t=trivial(L)
set tm6=parameter(x(t),lambda(t)-[0,0,0,0,0,6],nu(t))
goodness(tm6,G)
Value: "Weakly fair"
theta_induce_irreducible(tm6,G)
Value:
1*parameter(x=1773,lambda=[2,2,0,1,0,2]/1,nu=[3,3,-1,2,-1,3]/2) [11]
\end{verbatim}
Other rows are interpreted similarly. \hfill\qed
\end{example}

\section{Appendix}\label{sec-appendix}

This appendix presents all the $55$ \emph{non-trivial} fully supported scattered representations in $\widehat{E_{6(2)}}^d$ according to their infinitesimal characters. Since each coordinate of any involved infinitesimal character is either $0$ or $1$, it follows from Lemma 2.2 of \cite{DW21-1} that each spin-lowest $K$-type must be u-small.

\begin{table}[H]
\centering
\caption{Infinitesimal character $[1,0,0,1,1,1]$ and $[1,0,1,1,0,1]$}
\begin{tabular}{lcccc}
 $\# x$ & $\lambda$ & $\nu$ & Spin LKTs  & $\# x^\prime$ \\
\hline
$1686$ & $[3,-4,-2,5,3,1]$ & $[1,-2,-2,3,1,1]$ & $[0,0,2,2,0,2]$ & \textbf{1687}\\

$1592$ & $[1,-1,-1,3,1,2]$ & $[0,-2,-\frac{3}{2},\frac{7}{2},0,\frac{3}{2}]$ & $[2,0,1,0,0,7]$ & \textbf{1612}

\end{tabular}
\label{table-EII-100111}
\end{table}

\begin{table}[H]
\centering
\caption{Infinitesimal character  $[1,1,0,1,1,1]$ and $[1,1,1,1,0,1]$}
\begin{tabular}{lcccc}
 $\# x$ & $\lambda$ & $\nu$ & Spin LKTs  & $\# x^\prime$ \\
\hline
$1539$ & $[4,1,-3,3,1,1]$ & $[5,1,-4,1,1,1]$ & $[0,4,0,0,0,0]$, $[0,4,1,0,0,1]$ & \textbf{1540} \\

$1415$ & $[3,2,-1,1,1,1]$ & $[5,\frac{3}{2},-\frac{7}{2},0,2,0]$ & $[0,0,0,0,4,8]$, $[0,0,1,0,4,9]$ &\textbf{1398}

\end{tabular}
\label{table-EII-110111}
\end{table}

\begin{table}[H]
\centering
\caption{Infinitesimal character $[0,1,1,0,1,1]$ and $[1,1,1,0,1,0]$}
\begin{tabular}{lcccc}
 $\# x$ & $\lambda$ & $\nu$ & Spin LKTs  & $\# x^\prime$ \\
\hline
$1761$ & $[1,6,4,-3,3,1]$ & $[-1,2,2,-1,1,1]$ & $[0,0,2,1,0,2]$, $[0,0,3,1,0,3]$ & \textbf{1763}\\
& & & $[0,1,1,2,0,1]$, $[0,1,2,2,0,2]$ & \\

$1722$ & $[-1,3,2,0,1,2]$ & $[-\frac{3}{2},2,\frac{3}{2},0,0,\frac{3}{2}]$ & $[1,0,1,0,0,6]$, $[1,0,2,0,0,7]$ & \textbf{1728}\\
& & & $[2,0,0,0,1,7]$, $[2,0,1,0,1,8]$ & \\

$874$ & $[-1,1,4,-2,2,1]$ & $[-\frac{5}{2},1,\frac{7}{2},-\frac{5}{2},1,1]$ & $[0,2,1,0,3,4]$, $[0,3,0,0,2,6]$ & \textbf{896}\\
& & & $[1,1,0,1,3,6]$ &

\end{tabular}
\label{table-EII-011011}
\end{table}

\begin{table}[H]
\centering
\caption{Infinitesimal character $[0,1,1,0,1,0]$}
\begin{tabular}{lccc}
 $\# x$ & $\lambda$ & $\nu$ & Spin LKTs   \\
\hline
$1561$& $[-1,1,3,-1,3,-1]$ & $[-\frac{3}{2},1,\frac{5}{2},-\frac{3}{2},\frac{5}{2},-\frac{3}{2}]$ & $[0,0,4,0,0,2]$, $[0,1,0,1,0,6]$  \\
& & & $[0,2,0,2,0,6]$, $[1,1,0,1,1,8]$  \\

$1502$ & $[-1,1,3,-1,3,-1]$& $[-\frac{3}{2},0,\frac{5}{2},-1,\frac{5}{2},-\frac{3}{2}]$ & $[2,0,2,0,2,2]$, $[2,1,0,1,2,0]$  \\
& & & $[3,0,0,0,3,2]$  \\
\end{tabular}
\label{table-EII-011010}
\end{table}

\begin{table}[H]
\centering
\caption{Infinitesimal character $[1,1,0,1,0,1]$}
\begin{tabular}{lcccc}
 $\# x$ & $\lambda$ & $\nu$ & Spin LKTs  & $\# x^\prime$ \\
\hline
$1787_{\clubsuit}$ & $[1,1,1,4,-1,2]$ & $[1,1,0,1,0,1]$& $[0,0,4,0,0,4]$,$[0,1,2,1,0,2]$ & \\
& & & $[0,2,0,2,0,0]$ & \\

$1773_{\clubsuit}$ & $[2,2,0,1,0,2]$ & $[\frac{3}{2},\frac{3}{2},-\frac{1}{2}, 1, -\frac{1}{2}, \frac{3}{2}]$ & $[0,0,1,0,0,5]$, $[1,0,1,0,1,7]$ &\\
& & & $[2,0,1,0,2,9]$ & \\

$1352$ & $[1,1,-1,3,-1,1]$ & $[1,1,-\frac{5}{2},\frac{7}{2},-\frac{5}{2},1] $ & $[0,0,4,0,0,4]$, $[0,1,1,1,0,7]$ & \\
& & & $[0,2,0,2,0,8]$ & \\

$1269$ & $[1,1,-1,3,-1,1]$ & $[1,0,-\frac{5}{2},4,-\frac{5}{2},1]$ & $[3,0,0,0,3,0]$, $[3,0,1,0,3,1]$ & \\

$1166$ & $[7,3,-2,1,-3,6]$ & $[3,1,-2,1,-2,3]$ & $[0,2,0,2,0,0]$, $[1,2,0,2,1,2]$ & \\

$977$ & $[2,1,-1,2,-1,2]$ & $[\frac{5}{2}, 1, -\frac{5}{2}, 2, -\frac{5}{2},\frac{5}{2}]$ & $[0,1,0,1,0,12]$, $[1,0,0,0,1,10]$ & \\

$964$ & $[3,1,0,1,-1,2]$ & $[3,0,-2,2,-\frac{5}{2},\frac{5}{2}]$ & $[0,0,1,1,4,3]$, $[0,1,0,0,5,5]$ & \textbf{953}

\end{tabular}
\label{table-EII-110101}
\end{table}

\begin{table}[H]
\centering
\caption{Infinitesimal character $[1,0,0,1,0,1]$}
\begin{tabular}{lcccc}
 $\# x$ & $\lambda$ & $\nu$ & Spin LKTs  & $\# x^\prime$ \\
\hline
$1787_{\clubsuit}$ & $[1,1,1,4,-1,2]$ &  $[1,0,0,1,0,1]$ & $[1,0,3,0,1,1]$, $[1,1,1,1,1,3]$&\\
       &                  &                  &  $[2,0,1,0,2,5]$ &  \\

$1782_{\clubsuit}$ &$[1,0,1,4,-1,2]$ & $[1,0,0,1,0,1]$    & $[0,0,1,0,0,9]$, $[1,0,3,0,1,1]$\\
       &                  &                 &$[1,1,1,1,1,3]$, $[2,0,1,0,2,5]$  & \\

$1746_{\clubsuit}$  &$[3,0,-1,3,-1,3]$ & $[\frac{3}{2},-\frac{1}{2},-\frac{1}{2},\frac{3}{2},-\frac{1}{2},\frac{3}{2}]$ & $[0,0,4,0,0,0], [0,1,0,1,0,8]$ & \\
        &                 &                       & $[0,2,0,2,0,4], [1,1,0,1,1,6]$ & \\

$1726_{\clubsuit}$  &$[3,-1,0,2,0,3]$ &  $[\frac{3}{2},-1,0,1,0,\frac{3}{2}]$&  $[2,0,2,0,2,0]$, $[2,1,0,1,2,2]$ &\\
        &                 &                 &  $[3,0,0,0,3,4]$  &\\

$1537^*$ & $[3,-3,-2,6,-1,1]$ & $[1,-2,-1,3,-1,1]$ &$[1,1,1,1,1,3]$, $[2,1,0,1,2,2]$ & \\

$1377_{\clubsuit}$ &$[1,-2,-1,5,-1,1]$ & $[0,-2,-\frac{1}{2},3,-\frac{1}{2},0]$ & $[0,0,1,0,0,9]$ & \\

$1268^*$ & $[1,0,-2,5,-2,1]$ & $[1,0,-2,3,-2,1]$ &
$[0,0,1,0,0,9]$, $[0,1,0,1,0,8]$ & \\

$1267^*$ & $[1,0,-2,5,-2,1]$ & $[1,0,-2,3,-2,1]$ & $[0,2,0,2,0,4]$, $[1,1,1,1,1,3]$ & \\

$850_{\clubsuit}$ & $[3,0,-2,3,-2,3]$ & $[2,0,-2,2,-2,2]$ & $[0,0,0,0,0,10]$ & \\
$559$ & $[2,-1,-2,4,-2,2]$ & $[1,-\frac{3}{2},-\frac{3}{2},\frac{5}{2},-\frac{3}{2},1]$ &$[1,0,1,1,2,4]$, $[2,1,0,1,2,2]$ &\\
 & & & $[2,1,1,0,1,4]$, $[3,0,0,0,3,4]$ & \\

$1649^*$& $[3,-1,0,4,-3,3]$ & $[1,-\frac{3}{2},0,\frac{5}{2},-\frac{3}{2},1]$   & $[ 2,0,2,1,0,2]$, $[2,1,1,0,1,4]$ & $\textbf{1645}^*$\\
& & & $[3,0,1,1,1,1]$, $[3,1,0,0,2,3]$ & \\

$1403^*$& $[3,-3,-1,4,0,1]$& $[1,-\frac{5}{2},-\frac{1}{2},\frac{5}{2},0,0]$ &$[0,0,2,0,3,3]$,$[1,0,1,0,4,2]$ & $\textbf{1371}^*$\\

$1205$ & $[4,-3,-1,4,-1,1]$ & $[\frac{5}{2},-\frac{3}{2},-\frac{3}{2},\frac{5}{2},-\frac{3}{2},1]$ &
$[0,1,2,0,2,2]$, $[0,3,0,0,0,6]$ & \textbf{1198} \\
& & & $[1,0,1,1,2,4]$, $[1,1,0,1,1,6]$ & \\

$1130$ & $[5,-2,-2,3,0,1]$ & $[3,-1,-2,2,-1,1]$ & $[4,0,0,2,0,0]$, $[4,0,1,0,1,2]$ &
\textbf{1123} \\

$1129$ & $[5,-2,-2,3,0,1]$ & $[3,-1,-2,2,-1,1]$ & $[0,0,2,0,3,3]$,$[0,2,0,0,1,7]$ & \textbf{1124}  \\

$1128$ & $[5,-2,-2,3,0,1]$ & $[3,-1,-2,2,-1,1]$ & $[1,1,1,1,1,3]$,$[1,2,0,1,0,5]$ & \textbf{1122}  \\

$958^*$ & $[2,-1,-1,4,-3,3]$ & $[\frac{3}{2},-\frac{1}{2},-\frac{3}{2},\frac{5}{2},-\frac{5}{2},2]$ &
$ [0,1,0,0,5,1]$, $[1,0,1,0,4,2]$ & $\textbf{956}^*$
\end{tabular}
\label{table-EII-100101}
\end{table}

\begin{table}[H]
\centering
\caption{Infinitesimal character $[1,1,1,0,1,1]$}
\begin{tabular}{lcccc}
 $\# x$ & $\lambda$ & $\nu$ & Spin LKTs  & $\# x^\prime$ \\
\hline
$1789_{\clubsuit}^*$ & $[1,1,2,0,2,1]$ & $[1,1,1,0,1,1]$&$[0,0,0,0,0,2]+n\beta$, $1\leq n\leq 4$ & \\

$1225$ & $[1,4,1,-1,1,1]$ & $[1, \frac{9}{2},1,-\frac{7}{2},1,1]$ & $[0,0,3,0,0,7]$, $[0,0,4,0,0,6]$ & \\
& & & $[0,1,2,1,0,8]$ & \\

$1154$ & $[1,3,2,-2,2,1]$ & $[1, 4, \frac{3}{2}, -4, \frac{3}{2}, 1]$ & $[4,0,0,0,4,0]$, $[4,0,1,0,4,1]$ & \\

$1624$ & $[1,1,4,-1,1,1]$ & $[1,1,4,-3,1,1]$ &
$[0,3,0,0,0,0]+n\beta$, $0\leq n\leq 2$ & \textbf{1623}\\

$1534$ & $[2,2,2,-1,1,1]$ & $[\frac{3}{2},\frac{3}{2},\frac{7}{2},-\frac{7}{2},2,0]$ &
$[0,0,0,0,3,7]+n\beta$, $0\leq n\leq 2$ & \textbf{1517}
\end{tabular}
\label{table-EII-111011}
\end{table}

\centerline{\scshape Acknowledgements} Dong is supported by the National Natural Science Foundation of China (grant 12171344).
We express our sincere gratitude to the following mathematicians: to Prof.~Adams for offering their list of special unipotent representations, to Prof.~Pand\v zi\'c for telling us the idea of deducing Theorem \ref{thm-main}, to Prof.~Vogan for guiding us through Section \ref{sec-coroots},
to Daniel Wong for helpful discussions pertaining to Example \ref{exam-weakly-good}, and to Jia-Jun Ma for showing us how to save the \texttt{atlas} screen outputs into a file. We sincerely thank an anonymous referee for giving us excellent suggestions.

\end{document}